\documentclass[10pt]{article} 
\usepackage[utf8]{inputenc}
\usepackage{geometry}
\usepackage{authblk}
\usepackage{amsfonts}
\usepackage{xurl}
\usepackage{hyperref}
\usepackage{amsmath}
\usepackage{amsthm}
\usepackage{pdflscape}
\usepackage{pgfplots}
\usepackage{mathrsfs}
\usepackage{breakurl}
\usepackage{indentfirst}
\usepackage{amssymb}
\usepackage{amssymb,amsmath,amsthm,mathrsfs,array,graphicx,bbm,enumerate,wasysym,dsfont}
\newtheorem{thm}{Theorem}[section]

\newtheorem{lem}[thm]{Lemma}

\newtheorem{quest}{Question}[section]

\usepackage{xcolor}
\theoremstyle{definition}
\newtheorem{defn}{Definition}

\newenvironment{rem}{ \noindent\textbf{Remark.}\itshape}{\par}

\numberwithin{equation}{section}

\title{A few observations around Gaussian domination and continuous symmetry breaking for spin O(N) model}

\author{Xiao Han}
\affil{EPFL}
\date{}
\begin{document}
\maketitle
\begin{abstract}
We investigate the notion of Gaussian domination for the spin $O(N)$ model on general finite graphs. {We begin by proving a general inequality for spin correlations under the assumption of Gaussian domination, which directly implies long-range order at low temperatures for graphs with bounded Green's function.} Usually, Gaussian domination is proved via reflection positivity, but this requires strict symmetries and is very rigid. In this article we {also} probe the boundaries of elementary methods for proving Gaussian domination. Although we did not find a way to get uniform bounds, we do offer new views for Gaussian domination at low and high temperatures for finite graphs, and a few counterexamples illustrating the interplay between correlation estimates and Gaussian domination and how local changes in the graph structure can affect the presence of Gaussian domination.  

\end{abstract}

\section{Introduction}
\label{sec1}
We aim to study the spin $O(N)$ model, a classical statistical mechanics model defined on a graph or lattice, in which each vertex is assigned a spin taking values on the unit sphere $\mathbb{S}^{N-1} \subset \mathbb{R}^N$ \cite{3,4,5,12,13,14}. In other words, each spin is a unit vector in $\mathbb{R}^N$. More precisely, we define the spin \( O(N) \) model on a finite, undirected, unweighted, connected graph $G=(V, E)$. For convenience, we label the vertices as \( V = \{1, 2, \dots, n\} \) with \( n \in \mathbb{N}^* \).
\begin{defn}[\textbf{Spin $O(N)$ model}]
The spin $O(N)$ model on a finite, undirected, unweighted, connected graph $G=(V, E)$ at inverse temperature $\beta \in [0, \infty)$ is a random configuration $\sigma = (\sigma_1, \dots, \sigma_n) \in (\mathbb{S}^{N-1})^{V}$ with probability measure 
\begin{equation}
\label{def}
d\mu_{G,N,\beta}(\sigma)=\frac{1}{Z_{G, N,\beta}}\exp({-\beta H_{G, N}(\sigma)})d\sigma,
\end{equation}
where the Hamiltonian
\begin{equation}
 H_{G,N}(\sigma):=-\sum\limits_{\{i,j\}\in E}\sigma_i \cdot \sigma_j,
\end{equation}
the partition function 
\begin{equation}
Z_{G, N,\beta}:=\int_{(\mathbb{S}^{N-1})^{V}}\exp({-\beta H_{G, N}(\sigma)})d\sigma,
\end{equation}
and $d\sigma = d\sigma_1 \dots d\sigma_n$ denotes the product measure generated by the Lebesgue measure on the sphere. We use \( \mathbb{E}_{\mu_{G,N,\beta}} \) to denote the expectation with respect to the probability measure \( \mu_{G,N,\beta} \).
\end{defn}

In dimensions $d > 1$ the nearest-neighbour spin $O(N)$ model exhibits a phase transition. For the spin $O(N)$ model, it is well-known that at high temperatures, or when the underlying lattice is one-dimensional ($d = 1$), the correlations of spins decay exponentially with distance \cite{15,16,17}.

The low-temperature regime is however much trickier to understand. The case $N = 1$ corresponds to the classical Ising model, which has been extensively investigated. In this setting, it is known that there is long range order at low temperatures, i.e. the persistance of constant order correlation over distant spins, if and only if the spatial dimension satisfies $d \geq 2$. 

For models with continuous symmetry ($N \geq 2$), the Mermin--Wagner theorem establishes that no spontaneous symmetry breaking can occur in two-dimensional systems \cite{10,11}. On the other hand, there is a Berezinskii–Kosterlitz–Thouless transition in two-dimensional systems when $N=2$, where the model exhibit power-law decay of correlations at low temperatures \cite{19,20,21}. For $N>2$, the existence of a BKT transition in two dimensions is still unknown. When $d \geq 3$, Fröhlich, Simon, and Spencer showed that a symmetry-breaking phase transition does occur for the spin $O(N)$ model on high-dimensional tori \cite{1,2}. Their approach is based on the powerful method of {reflection positivity}, which relies on the presence of strong symmetry in the underlying graph (e.g., periodic boundary conditions on a torus). A key component of their method is the concept of {Gaussian domination} (see Definition \ref{GD} for details), which can be used to establish long-range order in the system.

However, the applicability of reflection positivity is largely restricted to settings with high symmetry, such as lattices with periodic boundary conditions. Thus, in this work, in order to gain a better understanding to low-temperature regime of $O(N)$ spin models and similar systems, we revisit the key notion of Gaussian domination using elementary means on general finite graphs. Our main contribution is presenting several counterexamples illustrating different ways in which these phenomena may fail to extend to such general settings, but on the way we also obtain a few general results relating Gaussian domination and spin-spin correlations.

\textbf{Acknowledgements}
 {The author would like to thank Juhan Aru and Sébastien Ott for their valuable feedback and comments. This work is supported by the Eccellenza grant 194648 of the Swiss National Science Foundation. The author is also a member of NCCR SwissMap.}

\section{Main results}
\label{sec2}
Let us first introduce the notion of Gaussian domination in the context of a general graph. Gaussian domination refers to an inequality that involves a modified version of the partition function.

\begin{defn}[\textbf{Modified partition function}] 
We define the modified partition function of spin $O(N)$ model with a shift $h = (h_1, \dots, h_n) \in (\mathbb{R}^N)^V$ by
\begin{equation}
Z^*_{G, N, \beta}(h) := \int_{(\mathbb{S}^{N-1})^V}\exp(-\frac{\beta}{2}\sum_{\{i, j\}\in E}\lVert\sigma_i+h_i-\sigma_j-h_j\rVert^2) d\sigma.
\end{equation}
\end{defn}
Note that $-\frac{1}{2}\lVert\sigma_i-\sigma_j \rVert^2=-1+\sigma_i \cdot \sigma_j$, thus $Z^*_{G, N, \beta}(0)$ is a constant multiple of $Z_{G, N,\beta}$.

\begin{defn}[\textbf{Gaussian domination}]
\label{GD}
Given $N\geq1$, a connected graph $G=(V, E)$, and $\beta \in [0, +\infty)$, we say that there is a (local) {Gaussian domination} in the spin $O(N)$ model on $G$ at inverse temperature $\beta$ if  $Z^*_{G, N, \beta}(h)$ has a local maximum point at $h=0$.  
\end{defn}
Note that by symmetry we have that $h=0$ is a stationary point of $Z^*_{G, N, \beta}(h)$, thus the Gaussian domination implies that its second-order partial derivatives in all directions are non-positive.

Now, given Gaussian domination, long-range order is stated and proved in a global averaged sense. In particular, it is shown that \cite{1,3}: for $d\geq 3$, $N\geq 1$ and  $G=\mathbb{T}^d_L$ (torus of side length $L$), there exists $\beta(d,N)>0$ and $C(d,N)>0$ such that for any $\beta>\beta(d,N)$ we have
$$
\frac{1}{|V(\mathbb{T}^d_L)|^2}\sum_{x,y\in V(\mathbb{T}^d_L)}\mathbb{E}_{\mu_{G,N,\beta}} \sigma_x \cdot \sigma_y>C(d,N).
$$

Our first observation is that Gaussian domination actually allows to directly conclude a bound on spin to spin correlations. To state this, let $\{X_k\}, k\geq 0$ be a simple random walk on $G$ and $\mathbb{P}_i$ and $\mathbb{E}_i$ be the probability law and the associated expectation for $X_k$ starting from $i\in V$. Further suppose $H_j:= \inf\{k\geq0 : X_k= j\}$ is the first hitting time of $j\in V$ by $X_k$ and $d(i)$ is the degree of $i\in V$. Finally, for any $i,j \in V$, $i\neq j$, $u_{ij}: V \to \mathbb{R}$ is defined by 
\begin{equation}
\label{uu}
u_{ij}(s):= \mathbb{E}_s\sum\limits_{k=0}^{H_j-1}1_{X_k=i}.
\end{equation}
If we view $j\in V$ as an absorbing state, this is exactly the Green's function on $G$. 
\begin{thm}
\label{GDE}
Given $N\geq1$, a connected graph $G=(V, E)$, and $\beta \in [0, +\infty)$, if there is a Gaussian domination in the spin $O(N)$ model on  $G$ at inverse temperature $\beta$, we have that $\mathbb{E}_{\mu_{G,N,\beta}} \sigma_x \cdot \sigma_y \geq 1-\frac{Nu_{xy}(x)}{2\beta d(x)}$ for any $x,y \in V$.
\end{thm}

 {When $d \ge 3$, we consider $G$ to be the torus $\mathbb{T}^d_{L}$, this theorem implies the uniform lower bound for the two-point function: there exist $c=c(d)>0$ and $\beta_c=\beta_c(d)>0$ such that for any $L >0$,  $\beta > \beta_c$ and $x, y \in G$ we have that $\mathbb{E}_{\mu_{G,N,\beta}} \sigma_x \cdot \sigma_y \geq c$. This was recently also obtained in \cite{23} for pairs of points that are not too distant, with a different lower bound.}

For Gaussian domination there are two parameters at play: the graph and the inverse temperature. A natural question is: under which conditions and at which temperatures does Gaussian domination hold? Further, how robust is it w.r.t. changing the conditions? 

First, for any fixed finite graph, Gaussian domination can be rigorously established in both the high-temperature and low-temperature regimes using elementary means. This might not be very surprising, but we were not able to pinpoint mathematical proofs in the literature.

\begin{thm}
\label{high}
Given $N\geq1$ and a connected graph $G=(V, E)$, there exists $\beta_s>0$, such that there is a Gaussian domination in the spin $O(N)$ model for $G$ and $0<\beta <\beta_s$.
\end{thm}

\begin{thm}
\label{low}
Given $N\geq 1$ and a connected graph $G=(V, E)$, there exists $\beta_d>0$, such that there is a Gaussian domination in the spin $O(N)$ model for $G$ and $\beta>\beta_d$.
\end{thm}

Conversely, we show that even rather modest local graph modifications can break the domination.

\begin{thm}
\label{addstar}
For any $N\geq 1$ and $\beta>0$, there exists $n_0\in \mathbb{N}$, such that for any graph $G=(V, E)$, we can find another graph $G'=(V', E')$, such that $G \subset G'$, $|V'|-|V| \leq n_0$, $|E'|- |E| \leq n_0$, and there is no Gaussian domination in the spin $O(N)$ model for $G'$ and $\beta$.
\end{thm}

As a byproduct we obtain the following zero temperature convergence at any finite graph; this is widely expected and accepted, but its proof seems not to be flashed out. Using reflection positivity in \cite{6}, \cite{7}, and \cite{22}, it has been shown for a growing sequence of symmetric graphs.

\begin{thm}
\label{Wcon}
Given $N\geq 2$ and a connected graph $G=(V, E)$, for a fixed point $1\in V$, we consider the spin $O(N)$ model $\sigma\in (\mathbb{S}^{N-1})^{V}$ on $G$ with inverse temperature $\beta$ and $\sigma_1$ rooted to be the north point (i.e. $\sigma_1^N=1$). Let $\gamma\in (\mathbb{R}^{N-1})^{V}$ denote an $(N-1)$-component vector-valued Gaussian Free Field on $G$, with $\gamma_1$ rooted at $0$. Then, as $\beta \to \infty$, the first $(N-1)$ components of the rescaled spin field converge in law to the Gaussian Free Field:
\begin{equation}
(\sqrt{\beta}\sigma^1, \sqrt{\beta}\sigma^2, \dots \sqrt{\beta}\sigma^{N-1} ) \stackrel{L}{\longrightarrow} \gamma.
\end{equation}
\end{thm}

Here we use superscripts to denote coordinates of points in $\mathbb{R}^N$. We prove also convergence of moments and rough quantitative estimates on the convergence rates, which unfortunately are not as good as one would hope - see Theorem \ref{mom} and Theorem \ref{Wbound} in Section \ref{sec4}.

Further, given that our counterexample  in Theorem \ref{addstar} involve graphs with sufficiently large degrees, it is natural to ask whether there exists a universal bound for the inverse temperature $\beta$ that guarantees Gaussian domination in certain types of graphs, particularly those with finite degree. However, the following result demonstrates that even for binary trees, such a universal bound does not exist.

\begin{thm}
\label{binary}
Let $N=1$, there exists $\beta_0>0$ such that for any $\beta>\beta_0$, there exists a binary tree $G=(V, E)$, such that there is no Gaussian domination in the spin $O(N)$ model on $G$ at inverse temperature $\beta$.
\end{thm}

Of course, the absence of Gaussian domination does not imply the absence of long range order. Motivated by this, we consider the following question, which may be of independent interest. In what follows, the renormalized Green's function between $x$ and $y$ refers to the quantity 
\[
\frac{u_{xy}(x)}{d(x)} = \frac{u_{yx}(y)}{d(y)}.
\]
(This identity can be quickly verified from Theorem~\ref{mom}.)

\begin{quest}
Given $N \in \mathbb{N}$ and $\epsilon > 0$, suppose that the renormalized Green's function of a graph $G$ is bounded above by $M$. Does there exist a universal critical inverse temperature $\beta_c = \beta_c(N, \epsilon, M)$ such that for all $x, y \in G$ and all $\beta > \beta_c$, the spin $O(N)$ model satisfies
\[
\mathbb{E}_{\mu_{G,N,\beta}} \sigma_x \cdot \sigma_y \geq 1 - \epsilon?
\]
\end{quest}
We note that if the renormalized Green's function is interpreted specifically as
$\frac{u_{xy}(x)}{d(x)}$
then the following result provides a counterexample:

\begin{thm}
\label{ppath}
Let $N = 1$. For any $M > 0$ and $\epsilon > 0$, there does not exist a universal $\beta_c = \beta_c(N, \epsilon, M)$ such that for every graph $G$, every pair of vertices $x, y \in G$ with $\frac{u_{xy}(x)}{d(x)} < M$, and every $\beta > \beta_c$, the inequality
\[
\mathbb{E}_{\mu_{G,N,\beta}} \sigma_x \cdot \sigma_y > 1 - \epsilon
\]
holds.
\end{thm}

Alternatively, if the renormalized Green's function is taken to refer more broadly to the maximum of $\frac{u_{ij}(i)}{d(i)}$ over all pairs of distinct vertices $i, j \in G$, the problem is still open.\\

The remainder of the paper is organized as follows. In Section~\ref{sec3}, we present the proof of Theorem~\ref{GDE}. Section~\ref{sec4} establishes Gaussian domination in the high-temperature and low-temperature regimes related to Theorem~\ref{high}, Theorem~\ref{low} and other byproducts. Sections~\ref{sec5}-\ref{sec7} provide three counterexamples corresponding to Theorems~\ref{addstar}, \ref{binary}, and \ref{ppath}, respectively.

\section{From Gaussian domination to long-range correlations}
\label{sec3}
In this section, we aim to prove Theorem \ref{GDE}. We will need the following lemma for the Green's function of $G$. Recall (\ref{uu}) for the definition of $u$. We remind the reader that the function $u$ is indexed by two subscripts.
\begin{lem}
\label{deltalem}
For any $i,j \in V$, $i\neq j$, we have
\begin{equation}
\Delta u_{ij}(s)=\left\{
\begin{aligned}
-d(i), &       &s=j,\\
d(i),  &       &s=i,\\
0,     &       &\text{otherwise},
\end{aligned}
\right.
\end{equation}
where $\Delta u_{ij}(s):=\sum\limits_{\{s,l\}\in E} (u_{ij}(s)-u_{ij}(l))$ is the Laplacian of $u_{ij}$ on $V$.
\end{lem}
\begin{proof}
Note that $\sum\limits_{s\in V} \Delta u_{ij}(s)=0$, to compute the Laplacian of $u_{ij}$ we only need to check that $\Delta u_{ij}(i)= d(i)$ and $\Delta u_{ij}(s)= 0$ for $s\in V\setminus \{i,j\}$ by the Markov property.
\end{proof}

We now begin the proof of Theorem \ref{GDE}. For the convenience of the reader, we restate it here.
\begin{thm}
\label{GDE2}
Given $N\geq1$, a connected graph $G=(V, E)$, and $\beta \in [0, +\infty)$, if there is a Gaussian domination in the spin $O(N)$ model on  $G$ at inverse temperature $\beta$, we have that $\mathbb{E}_{\mu_{G,N,\beta}} \sigma_x \cdot \sigma_y \geq 1-\frac{Nu_{xy}(x)}{2\beta d(x)}$ for any $x,y \in V$.
\end{thm}

\begin{proof}
We only need to consider the case $x\neq y$. Fix a direction vector $v= (v_1, v_2, \dots, v_n) \in (\mathbb{R}^N)^V$ ($v \neq 0$). We define $Z_v(\eta) : \mathbb{R} \to \mathbb{R}$ by $Z_v(\eta):= Z^*_{G, N, \beta}(\eta v)$ for any $\eta \in \mathbb{R}$. By Gaussian domination and symmetry, we have $Z_{v}'(0)=0$ and $Z_{v}''(0)\leq0$. On the other hand, since
\begin{equation}
Z_v(\eta)=\int_{(\mathbb{S}^{N-1})^V}\exp(-\frac{\beta}{2}\sum\limits_{\{i, j\}\in E}\lVert\sigma_i+\eta v_i-\sigma_j- \eta v_j\lVert^2) d\sigma,
\end{equation}
by the dominated convergence theorem, we have
\begin{equation}
Z_v'(\eta)= -\beta\int_{(\mathbb{S}^{N-1})^V}[\sum_{\{i, j\}\in E}(v_i-v_j)\cdot(\sigma_i+\eta v_i-\sigma_j- \eta v_j)]\exp(-\frac{\beta}{2}\sum\limits_{\{i, j\}\in E}\lVert\sigma_i+\eta v_i-\sigma_j- \eta v_j\rVert^2) d\sigma.
\end{equation}
Furthermore,
\begin{equation}
\label{de}
\begin{aligned}
Z_v''(0)=&  \beta^2\int_{(\mathbb{S}^{N-1})^V}[\sum_{\{i, j\}\in E}(v_i-v_j)\cdot(\sigma_i-\sigma_j)]^2\exp(-\frac{\beta}{2}\sum\limits_{\{i, j\}\in E}\lVert\sigma_i-\sigma_j\rVert^2) d\sigma \\
&-\beta\int_{(\mathbb{S}^{N-1})^V}[\sum_{\{i, j\}\in E}\lVert v_i-v_j \rVert^2]\exp(-\frac{\beta}{2}\sum\limits_{\{i, j\}\in E}\lVert\sigma_i-\sigma_j\rVert^2) d\sigma.
\end{aligned}
\end{equation}
For any unit vector $e\in \mathbb{R}^N$, if we take $v=\frac{1}{d(x)}u_{xy}e$, we will have
\begin{equation}
\label{de2}
\begin{aligned}
Z_v''(0)=&  \beta^2\int_{(\mathbb{S}^{N-1})^V}[\sum_{i\in V}\frac{1}{d(x)}\Delta u_{xy}(i)e\cdot \sigma_i]^2\exp(-\frac{\beta}{2}\sum\limits_{\{i, j\}\in E}\lVert\sigma_i-\sigma_j \rVert^2) d\sigma \\
&-\beta\int_{(\mathbb{S}^{N-1})^V}[\sum_{i\in V}\frac{1}{d(x)^2}\Delta u_{xy}(i)\cdot u_{xy}(i)]\exp(-\frac{\beta}{2}\sum\limits_{\{i, j\}\in E}\lVert\sigma_i-\sigma_j \rVert^2) d\sigma.\\
=& \beta^2\int_{(\mathbb{S}^{N-1})^V}[e\cdot \sigma_x- e\cdot \sigma_y]^2\exp(-\frac{\beta}{2}\sum\limits_{\{i, j\}\in E}\lVert\sigma_i-\sigma_j \rVert^2) d\sigma\\
&-\frac{\beta}{d(x)}\int_{(\mathbb{S}^{N-1})^V}[u_{xy}(x)-u_{xy}(y)]\exp(-\frac{\beta}{2}\sum\limits_{\{i, j\}\in E}\lVert\sigma_i-\sigma_j \rVert^2) d\sigma.
\end{aligned}
\end{equation}
Note that $Z_v''(0)\leq 0$ and $u_{xy}(y)=0$, we have
\begin{equation}
\int_{(\mathbb{S}^{N-1})^V}[e \cdot (\sigma_x-\sigma_y)]^2\exp(-\frac{\beta}{2}\sum\limits_{\{i, j\}\in E}\lVert\sigma_i-\sigma_j \rVert^2) d\sigma \leq \frac{u_{xy}(x)}{\beta d(x)}\int_{(\mathbb{S}^{N-1})^V}\exp(-\frac{\beta}{2}\sum\limits_{\{i, j\}\in E}\lVert\sigma_i-\sigma_j \rVert^2) d\sigma.
\end{equation}
Thus, if we let \( e \) range over the standard orthonormal basis \( \{e_1, e_2, \dots, e_N\} \) of \( \mathbb{R}^N \), and sum the inequality above over all such \( e \), we obtain the following result.

\begin{equation}
\mathbb{E}_{\mu_{G,N,\beta}} \lVert \sigma_x - \sigma_y \rVert^2\leq \frac{N u_{xy}(x)}{\beta d(x)},
\end{equation}
which means that
\begin{equation}
\mathbb{E}_{\mu_{G,N,\beta}} \sigma_x \cdot \sigma_y \geq 1-\frac{N u_{xy}(x)}{2\beta d(x)}.
\end{equation}
\end{proof}

\section{Phase transition for Gaussian domination}
\label{sec4}
In this section, we will work on the phase transition associated with Gaussian domination on a fixed finite graph.

\subsection{Gaussian domination at high temperature}

We first prove Theorem~\ref{high}, which establishes Gaussian domination in the high-temperature regime.
\begin{thm}
Given $N\geq1$ and a connected graph $G=(V, E)$, there exists $\beta_s>0$, such that there is a Gaussian domination in the spin $O(N)$ model for $G$ and $0<\beta <\beta_s$.
\end{thm}
\begin{proof}
We will use the computation in the proof of Theorem \ref{GDE2}. We first fixed the direction vector $v= (v_1, v_2, \dots, v_n) \in (\mathbb{R}^N)^V$ ($v \neq 0$). Note that if $v_x (x\in V)$ are all the same, the modified partition function would be a constant. Otherwise by $(\ref{de})$ it will be enough to show that for any $\sigma \in (\mathbb{S}^{N-1})^V$,
\begin{equation}
\label{ineq}
\beta^2[\sum_{\{i, j\}\in E}(v_i-v_j)\cdot(\sigma_i-\sigma_j)]^2 < \beta \sum_{\{i, j\}\in E}\lVert v_i-v_j \rVert^2,
\end{equation}
so that the second-order derivative of the modified partition function is negative, and we have the Gaussian domination. By the Cauchy-Schwarz inequality, we have
\begin{equation}
[\sum_{\{i, j\}\in E}(v_i-v_j)\cdot(\sigma_i-\sigma_j)]^2 \leq [\sum_{\{i, j\}\in E}\lVert\sigma_i-\sigma_j\rVert^2][\sum_{\{i, j\}\in E}\lVert v_i-v_j \rVert^2].
\end{equation}
On the other hand, note that $\lVert\sigma_i-\sigma_j\rVert^2 \leq 4$, so we take $\beta_s=\frac{1}{8|E|}$ and since $\sum\limits_{\{i, j\}\in E}\lVert v_i-v_j \rVert^2>0$, (\ref{ineq}) holds and there is a Gaussian domination for $G$ and $0<\beta <\beta_s$.
\end{proof}

\subsection{Gaussian domination at low temperature}
We now turn to the low-temperature regime, beginning with the proofs of several auxiliary results (Theorems~\ref{Wcon1}–\ref{Wbound}), which are essential for establishing Gaussian domination.

The first auxiliary result is Theorem \ref{Wcon}. For the convenience of the reader, we give out the law of the vector-valued Gaussian free field and restate the theorem here.
\begin{defn}[\textbf{Vector-valued Gaussian free field}]
The $N-$ component vector-valued Gaussian free field with free boundary condition on a finite undirected unweighted connected graph $G=(V, E)$ is a Gaussian vector $\gamma = (\gamma_1, \dots, \gamma_n) \in (\mathbb{R}^{N})^{V}$ with probability measure
\begin{equation}
\mathbb{P}(d\gamma) \propto \exp(-\frac{1}{2} \sum_{\{i,j\}\in E} ||\gamma_i-\gamma_j||^2)d\gamma,
\end{equation}
where $d\gamma=d\gamma_1 d\gamma_2 \dots d\gamma_n$ denotes the Lebesgue measure on $(\mathbb{R}^{N})^{V}$.
\end{defn}

\begin{thm}
\label{Wcon1}
Given $N\geq 2$ and a connected graph $G=(V, E)$, for a fixed point $1\in V$, we consider the spin $O(N)$ model $\sigma\in (\mathbb{S}^{N-1})^{V}$ on $G$ with inverse temperature $\beta$ and $\sigma_1$ rooted to be the north point (i.e. $\sigma_1^N=1$). Let $\gamma\in (\mathbb{R}^{N-1})^{V}$ denote an $(N-1)$-component vector-valued Gaussian Free Field on $G$, with $\gamma_1$ rooted at $0$. Then, as $\beta \to \infty$, the first $(N-1)$ components of the rescaled spin field converge in law to the Gaussian Free Field:
\begin{equation}
(\sqrt{\beta}\sigma^1, \sqrt{\beta}\sigma^2, \dots \sqrt{\beta}\sigma^{N-1} ) \stackrel{L}{\longrightarrow} \gamma.
\end{equation}
\end{thm}
\begin{proof}
 Note that $\sigma_1= (0, 0, \dots, 1)$ and $\gamma_1=(0, 0, \dots,0)$.
Firstly, we point out that given the endpoints of intervals $a_x^{l} < b_x^{l},$ where $a_x^l,b_x^l\in \mathbb{R}$, $x\in V\setminus\{1\}$ and $l\in \{1,2,\dots, N-1\}$, we have that as $\beta \to \infty$,
\begin{equation}
\label{tangent}
\begin{aligned}
&\int_{(\sigma_2,\sigma_3, \dots, \sigma_n)\in (\mathbb{S}^{N-1})^{V\setminus\{1\}}} (\prod_{x, l } 1_{a_x^{l} < \sqrt{\beta}\sigma_x^l< b_x^{l}})
(\prod_{x} 1_{\sigma_x^N>0})e^{-\frac{1}{2}\sum\limits_{\{i, j\}\in E}\lVert\sqrt{\beta}\sigma_i-\sqrt{\beta}\sigma_j\rVert^2} d(\sqrt{\beta}\sigma_2)\dots d(\sqrt{\beta}\sigma_n) \\
\to &\int_{(\gamma_2,\gamma_3,\dots,\gamma_n) \in (\mathbb{R}^{N-1})^{V\setminus\{1\}}} (\prod_{x, l } 1_{a_x^{l} < \gamma_x^l< b_x^{l}}) e^{-\frac{1}{2}\sum\limits_{\{i, j\}\in E}\lVert\gamma_i-\gamma_j\rVert^2}d\gamma_2d\gamma_3\dots d\gamma_n.
\end{aligned}
\end{equation}
Here $d(\sqrt{\beta}\sigma_i)$ denotes the Lebesgue measure on the rescaled sphere. To prove (\ref{tangent}), we observe that although the spherical region near the north point varies with the rescaling, its projection onto the tangent space at the north point remains fixed, and that the Lebesgue measure on the spherical region converges to the Lebesgue measure on this fixed projection as the sphere is rescaled to infinite radius. Let \(\pi\) denote the projection onto the tangent space of the sphere at the north point, $d\pi(\sqrt{\beta}\sigma_i)$ denote the Lebesgue measure on this tangent space, we have that
\begin{equation}
\begin{aligned}
&\lim_{\beta\to \infty}\int_{(\sigma_2,\sigma_3, \dots, \sigma_n)\in (\mathbb{S}^{N-1})^{V\setminus\{1\}}} (\prod_{x, l } 1_{a_x^{l} < \sqrt{\beta}\sigma_x^l< b_x^{l}})
(\prod_{x} 1_{\sigma_x^N>0})e^{-\frac{1}{2}\sum\limits_{\{i, j\}\in E}\lVert\sqrt{\beta}\sigma_i-\sqrt{\beta}\sigma_j\rVert^2} \prod_{i=2}^n d(\sqrt{\beta}\sigma_i) \\
= &\lim_{\beta\to \infty} \int_{(\sigma_2,\sigma_3, \dots, \sigma_n)\in (\mathbb{S}^{N-1})^{V\setminus\{1\}}}(\prod_{x, l } 1_{a_x^{l} < \sqrt{\beta}\sigma_x^l< b_x^{l}})
(\prod_{x} 1_{\sigma_x^N>0})e^{-\frac{1}{2}\sum\limits_{\{i, j\}\in E}\lVert\sqrt{\beta}\sigma_i-\sqrt{\beta}\sigma_j\rVert^2} \prod_{i=2}^n(d\pi(\sqrt{\beta}\sigma_i))\\
= &\lim_{\beta\to \infty}\int_{\substack{(\gamma_2,\gamma_3,\dots,\gamma_n) \in \\
(\mathbb{R}^{N-1} \cap B(0,\sqrt{\beta}))^{V\setminus\{1\}}}}  (\prod_{x, l } 1_{a_x^{l} < \gamma_x^l< b_x^{l}})\cdot e^{-\frac{1}{2}\sum\limits_{\{i, j\}\in E}(\lVert\gamma_i-\gamma_j \rVert^2+(\sqrt{\beta-\lVert \gamma_i \rVert^2}-\sqrt{\beta -\lVert \gamma_j \rVert^2})^2)}\prod_{i=2}^n d\gamma_i\\
= & \int_{(\gamma_2,\gamma_3,\dots,\gamma_n) \in (\mathbb{R}^{N-1})^{V\setminus\{1\}}} (\prod_{x, l } 1_{a_x^{l} < \gamma_x^l< b_x^{l}}) e^{-\frac{1}{2}\sum\limits_{\{i, j\}\in E}\lVert\gamma_i-\gamma_j \rVert^2}d\gamma_2d\gamma_3\dots d\gamma_n.
\end{aligned}
\end{equation}
The last equality follows from the fact that $\sqrt{\beta-\lVert \gamma_i \rVert^2}-\sqrt{\beta -\lVert \gamma_j \rVert^2} \to 0$ as $\beta \to \infty$ for fixed $\gamma_i$ and $\gamma_j$.

Secondly, we want to prove that as $\beta \to \infty$
\begin{equation}
\label{upsphere}
\begin{aligned}
& \int_{(\sigma_2,\sigma_3, \dots, \sigma_n)\in (\mathbb{S}^{N-1})^{V\setminus\{1\}}} e^{-\frac{1}{2}\sum\limits_{\{i, j\}\in E}\lVert\sqrt{\beta}\sigma_i-\sqrt{\beta}\sigma_j\rVert^2} d(\sqrt{\beta}\sigma_2)\dots d(\sqrt{\beta}\sigma_n) \\
\to & \int_{(\gamma_2,\gamma_3,\dots,\gamma_n) \in (\mathbb{R}^{N-1})^{V\setminus\{1\}}} e^{-\frac{1}{2}\sum\limits_{\{i, j\}\in E}\lVert\gamma_i-\gamma_j \rVert^2}d\gamma_2d\gamma_3\dots d\gamma_n.
\end{aligned}
\end{equation}
To show this, we note that for any \(\delta > 0\), there exists \(\epsilon > 0\) such that for any point \(z = (z^1, z^2, \dots, z^N) \in \mathbb{S}^{N-1}\) with \(z^N > 1 - \epsilon\), the Radon–Nikodym derivative \(\frac{dz}{d\pi(z)}\) is bounded above by \(1 + \delta\). Similarly, \(dz\) denotes the Lebesgue measure on the sphere and \(d\pi(z)\) is the Lebesgue measure on the tangent space. Notice that
\[
\frac{d(\sqrt{\beta} z)}{d\pi(\sqrt{\beta} z)} = \frac{d z}{d \pi(z)}.
\]
Taking \( z \) to be \(\sigma_2, \sigma_3, \dots, \sigma_n\), we have

\begin{equation}
\begin{aligned}
&\int_{(\sigma_2,\sigma_3, \dots, \sigma_n)\in (\mathbb{S}^{N-1})^{V\setminus\{1\}}} 
(\prod_{x} 1_{\sigma_x^N>1-\epsilon})e^{-\frac{1}{2}\sum\limits_{\{i, j\}\in E}\lVert\sqrt{\beta}\sigma_i-\sqrt{\beta}\sigma_j\rVert^2} d(\sqrt{\beta}\sigma_2)\dots d(\sqrt{\beta}\sigma_n) \\
\leq & (1+\delta)^{n-1}\int_{(\sigma_2,\sigma_3, \dots, \sigma_n)\in (\mathbb{S}^{N-1})^{V\setminus\{1\}}} 
(\prod_{x} 1_{\sigma_x^N>1-\epsilon})e^{-\frac{1}{2}\sum\limits_{\{i, j\}\in E}\lVert\sqrt{\beta}\sigma_i-\sqrt{\beta}\sigma_j\rVert^2} d(\pi(\sqrt{\beta}\sigma_2))\dots d(\pi(\sqrt{\beta}\sigma_n)) \\
\leq & (1+\delta)^{n-1} \int_{(\gamma_2,\gamma_3,\dots,\gamma_n) \in (\mathbb{R}^{N-1}\cap B(0,\sqrt{\beta}))^{V\setminus\{1\}}}e^{-\frac{1}{2}\sum\limits_{\{i, j\}\in E}(\lVert\gamma_i-\gamma_j \rVert^2+(\sqrt{\beta-\lVert \gamma_i \rVert^2}-\sqrt{\beta -\lVert \gamma_j \rVert^2})^2)}d\gamma_2\dots d\gamma_n\\
\leq & (1+\delta)^{n-1} \int_{(\gamma_2,\gamma_3,\dots,\gamma_n) \in (\mathbb{R}^{N-1})^{V\setminus\{1\}}}e^{-\frac{1}{2}\sum\limits_{\{i, j\}\in E}\lVert\gamma_i-\gamma_j \rVert^2}d\gamma_2\dots d\gamma_n.
\end{aligned}
\end{equation}
On the other hand, we note that for any $x \in V\setminus \{1\}$, if $\sigma_x^N \leq 1-\epsilon$, then there exists $\{i,j\}\in E$ such that $|\sigma_i^N-\sigma_j^N|>\frac{1}{n} |\sigma_1^N-\sigma_x^N|=\frac{\epsilon}{n}$ as the graph $G=(V,E)$ is connected. Thus
\begin{equation}
\label{connect}
\begin{aligned}
&\int_{(\sigma_2,\sigma_3, \dots, \sigma_n)\in (\mathbb{S}^{N-1})^{V\setminus\{1\}}} 
1_{\sigma_x^N\leq 1-\epsilon} \cdot e^{-\frac{1}{2}\sum\limits_{\{i, j\}\in E}\lVert\sqrt{\beta}\sigma_i-\sqrt{\beta}\sigma_j\rVert^2} d(\sqrt{\beta}\sigma_2)\dots d(\sqrt{\beta}\sigma_n) \\
\leq &\int_{(\sigma_2,\sigma_3, \dots, \sigma_n)\in (\mathbb{S}^{N-1})^{V\setminus\{1\}}} 
 e^{-\frac{\beta\epsilon^2}{2n^2}} d(\sqrt{\beta}\sigma_2)\dots d(\sqrt{\beta}\sigma_n)\\
 =&\exp(-\frac{\beta\epsilon^2}{2n^2}) \int_{(\sigma_2,\sigma_3, \dots, \sigma_n)\in (\mathbb{S}^{N-1})^{V\setminus\{1\}}} 
  d(\sqrt{\beta}\sigma_2)\dots d(\sqrt{\beta}\sigma_n),
\end{aligned}
\end{equation}
which naturally converges to \(0\) since the surface area of the rescaled sphere grows polynomially in \(\beta\), whereas \(\exp\left(-\frac{\beta \epsilon^2}{2 n^2}\right)\) decays exponentially in \(\beta\). Combining this with the fact that
\[
\prod_{x \in V \setminus \{1\}} 1_{\sigma_x^N > 1 - \epsilon} + \sum_{x \in V \setminus \{1\}} 1_{\sigma_x^N \leq 1 - \epsilon} \geq 1,
\]
we conclude that for any \(\delta > 0\), 

\begin{equation}
\begin{aligned}
&\limsup\limits_{\beta\to \infty} \int_{(\sigma_2,\sigma_3, \dots, \sigma_n)\in (\mathbb{S}^{N-1})^{V\setminus\{1\}}} e^{-\frac{1}{2}\sum\limits_{\{i, j\}\in E}\lVert\sqrt{\beta}\sigma_i-\sqrt{\beta}\sigma_j\rVert^2} d(\sqrt{\beta}\sigma_2)\dots d(\sqrt{\beta}\sigma_n) \\
\leq &  \limsup\limits_{\beta\to \infty}\int_{(\sigma_2,\sigma_3, \dots, \sigma_n)\in (\mathbb{S}^{N-1})^{V\setminus\{1\}}}(\prod\limits_{x\in V\setminus\{1\}} 1_{\sigma_x^N>1-\epsilon}) e^{-\frac{1}{2}\sum\limits_{\{i, j\}\in E}\lVert\sqrt{\beta}\sigma_i-\sqrt{\beta}\sigma_j\rVert^2} \prod_{i=2}^n d(\sqrt{\beta}\sigma_i) \\
&+\limsup\limits_{\beta\to \infty}\int_{(\sigma_2,\sigma_3, \dots, \sigma_n)\in (\mathbb{S}^{N-1})^{V\setminus\{1\}}}(\sum\limits_{x\in V\setminus\{1\}}1_{\sigma_x^N\leq 1-\epsilon}) e^{-\frac{1}{2}\sum\limits_{\{i, j\}\in E}\lVert\sqrt{\beta}\sigma_i-\sqrt{\beta}\sigma_j\rVert^2} \prod_{i=2}^n d(\sqrt{\beta}\sigma_i)\\
\leq & (1+\delta)^{n-1} \int_{(\gamma_2,\gamma_3,\dots,\gamma_n) \in (\mathbb{R}^{N-1})^{V\setminus\{1\}}}e^{-\frac{1}{2}\sum\limits_{\{i, j\}\in E}\lVert\gamma_i-\gamma_j \rVert^2}d\gamma_2\dots d\gamma_n + (n-1)\cdot 0 \\
= & (1+\delta)^{n-1} \int_{(\gamma_2,\gamma_3,\dots,\gamma_n) \in (\mathbb{R}^{N-1})^{V\setminus\{1\}}}e^{-\frac{1}{2}\sum\limits_{\{i, j\}\in E}\lVert\gamma_i-\gamma_j \rVert^2}d\gamma_2\dots d\gamma_n.
\end{aligned}
\end{equation}
Furthermore, from (\ref{tangent}) we know that
\begin{equation}
\begin{aligned}
&\liminf\limits_{\beta\to \infty} \int_{(\sigma_2,\sigma_3, \dots, \sigma_n)\in (\mathbb{S}^{N-1})^{V\setminus\{1\}}} e^{-\frac{1}{2}\sum\limits_{\{i, j\}\in E}\lVert\sqrt{\beta}\sigma_i-\sqrt{\beta}\sigma_j\rVert^2} d(\sqrt{\beta}\sigma_2)\dots d(\sqrt{\beta}\sigma_n) \\
\geq & \int_{(\gamma_2,\gamma_3,\dots,\gamma_n) \in (\mathbb{R}^{N-1})^{V\setminus\{1\}}}e^{-\frac{1}{2}\sum\limits_{\{i, j\}\in E}\lVert\gamma_i-\gamma_j \rVert^2}d\gamma_2\dots d\gamma_n,
\end{aligned}
\end{equation}
thus we have already proved (\ref{upsphere}). From (\ref{tangent}) and (\ref{upsphere}), it is straightforward to see the weak convergence.
\end{proof}

We now turn our attention to theorems concerning the convergence of moments and to obtaining rough quantitative estimates of the corresponding convergence rates.
\begin{thm}
\label{mom}
Given $N\geq 1$ and a connected graph $G=(V, E)$, for any $x,y\in V$ we have that as $\beta \to \infty$,
\begin{equation}
\beta \cdot \mathbb{E}_{\mu_{G,N,\beta}} \lVert \sigma_x - \sigma_y \rVert^2 \to (N-1)\frac{u_{xy}(x)}{d(x)}.
\end{equation}
\end{thm}
\begin{proof}
We first focus on the case where $N>1$. It would be enough to prove the case for $y=1$ and $x=2$.  Note that the moments remain the same if we root $\sigma_1$ to be the north point $(0,\dots, 0, 1)$. Thus if we denote the law of the rooted spin $O(N)$ model by $\mu_{G, N,\beta}^r$, we have that

\begin{equation}
\label{sum}
\begin{aligned}
   \beta \cdot \mathbb{E}_{\mu_{G,N,\beta}} \lVert\sigma_1 - \sigma_2 \rVert^2 =   &   \beta \cdot \mathbb{E}_{\mu^r_{G,N,\beta}} \lVert\sigma_1 - \sigma_2 \rVert^2\\
=   &   \mathbb{E}_{\mu^r_{G,N,\beta}} \sum_{s=1}^{N-1}(\sqrt{\beta}\sigma_2^s)^2 + \beta \mathbb{E}_{\mu^r_{G,N,\beta}}(1- \sigma_2^N)^2.
\end{aligned}
\end{equation}
Recall that $\gamma\in (\mathbb{R}^{N-1})^{V}$ is an $(N-1)$-component vector-valued Gaussian Free Field on $G$, with $\gamma_1$ rooted at $0$. Similar to (\ref{tangent}) we have that for any $s\in \{1, 2, \dots, N-1\}$ and $k\in \mathbb{N}^*$
\begin{equation}
\begin{aligned}
\int_{(\sigma_2,\sigma_3, \dots, \sigma_n)\in (\mathbb{S}^{N-1})^{V\setminus\{1\}}}& (\prod_{x, l } 1_{a_x^{l} < \sqrt{\beta}\sigma_x^l< b_x^{l}})
(\prod_{x} 1_{\sigma_x^N>0})(\sqrt{\beta}\sigma_2^s)^{2k} \\
& \cdot e^{-\frac{1}{2}\sum\limits_{\{i, j\}\in E}\lVert\sqrt{\beta}\sigma_i-\sqrt{\beta}\sigma_j\rVert^2} d(\sqrt{\beta}\sigma_2)\dots d(\sqrt{\beta}\sigma_n) \\
\to \int_{(\gamma_2,\gamma_3,\dots,\gamma_n) \in (\mathbb{R}^{N-1})^{V\setminus\{1\}}} & (\prod_{x, l } 1_{a_x^{l} < \gamma_x^l< b_x^{l}}) (\gamma_2^s)^{2k} \cdot e^{-\frac{1}{2}\sum\limits_{\{i, j\}\in E}\lVert\gamma_i-\gamma_j \rVert^2}d\gamma_2d\gamma_3\dots d\gamma_n,
\end{aligned}
\end{equation}
On the other hand, we could also use the same technique of (\ref{upsphere}) to show that for any $\delta>0$, there exists $\epsilon>0$ such that
\begin{equation}
\begin{aligned}
&\int_{(\sigma_2,\sigma_3, \dots, \sigma_n)\in (\mathbb{S}^{N-1})^{V\setminus\{1\}}} 
(\prod_{x} 1_{\sigma_x^N>1-\epsilon})(\sqrt{\beta}\sigma_2^s)^{2k} \cdot e^{-\frac{1}{2}\sum\limits_{\{i, j\}\in E}\lVert\sqrt{\beta}\sigma_i-\sqrt{\beta}\sigma_j\rVert^2} d(\sqrt{\beta}\sigma_2)\dots d(\sqrt{\beta}\sigma_n) \\
\leq & (1+\delta)^{n-1} \int_{(\gamma_2,\gamma_3,\dots,\gamma_n) \in (\mathbb{R}^{N-1})^{V\setminus\{1\}}}(\gamma_2^s)^{2k} \cdot e^{-\frac{1}{2}\sum\limits_{\{i, j\}\in E}\lVert\gamma_i-\gamma_j \rVert^2}d\gamma_2\dots d\gamma_n.
\end{aligned}
\end{equation}
and 
\begin{equation}
\begin{aligned}
&\int_{(\sigma_2,\sigma_3, \dots, \sigma_n)\in (\mathbb{S}^{N-1})^{V\setminus\{1\}}} 
1_{\sigma_x^N\leq 1-\epsilon} \cdot (\sqrt{\beta}\sigma_2^s)^{2k} \cdot e^{-\frac{1}{2}\sum\limits_{\{i, j\}\in E}\lVert\sqrt{\beta}\sigma_i-\sqrt{\beta}\sigma_j\rVert^2} d(\sqrt{\beta}\sigma_2)\dots d(\sqrt{\beta}\sigma_n) \\
\leq &\int_{(\sigma_2,\sigma_3, \dots, \sigma_n)\in (\mathbb{S}^{N-1})^{V\setminus\{1\}}} 
 \beta^k \exp(-\frac{\beta\epsilon^2}{2n^2}) d(\sqrt{\beta}\sigma_2)\dots d(\sqrt{\beta}\sigma_n),
\end{aligned}
\end{equation}
which converges to $0$. Thus as (\ref{upsphere}) we have that
\begin{equation}
\begin{aligned}
& \int_{(\sigma_2,\sigma_3, \dots, \sigma_n)\in (\mathbb{S}^{N-1})^{V\setminus\{1\}}} (\sqrt{\beta}\sigma_2^s)^{2k} \cdot e^{-\frac{1}{2}\sum\limits_{\{i, j\}\in E}\lVert\sqrt{\beta}\sigma_i-\sqrt{\beta}\sigma_j\rVert^2} d(\sqrt{\beta}\sigma_2)\dots d(\sqrt{\beta}\sigma_n) \\
\to & \int_{(\gamma_2,\gamma_3,\dots,\gamma_n) \in (\mathbb{R}^{N-1})^{V\setminus\{1\}}} (\gamma_2^s)^{2k}  \cdot e^{-\frac{1}{2}\sum\limits_{\{i, j\}\in E}\lVert\gamma_i-\gamma_j \rVert^2}d\gamma_2d\gamma_3\dots d\gamma_n.
\end{aligned}
\end{equation}
Combine it with (\ref{upsphere}) we have 
\begin{equation}
\label{moment}
\mathbb{E}_{\mu^r_{G,N,\beta}} (\sqrt{\beta}\sigma_2^s)^{2k} \to \mathbb{E}(\gamma_2^s)^{2k}.
\end{equation}
Take $k=1$ in (\ref{moment}) we know that $\mathbb{E}_{\mu^r_{G,N,\beta}} (\sqrt{\beta}\sigma_2^s)^2 \to \mathbb{E}(\gamma_2^s)^2 = \frac{u_{xy}(x)}{d(x)}$. The last inequality follows from, e.g., Definition 1.36 and the remarks below the definition in \cite{9}.
By (\ref{sum}), it remains to show that $\beta \mathbb{E}_{\mu^r_{G,N,\beta}}(1- \sigma_2^N)^2\to 0$. Similar to (\ref{connect}) we have 
\begin{equation}
\begin{aligned}
&\int_{(\sigma_2,\sigma_3, \dots, \sigma_n)\in (\mathbb{S}^{N-1})^{V\setminus\{1\}}} 
1_{\sigma_2^N \leq 0} \cdot \beta (1- \sigma_2^N)^2 \cdot e^{-\frac{1}{2}\sum\limits_{\{i, j\}\in E}\lVert\sqrt{\beta}\sigma_i-\sqrt{\beta}\sigma_j\rVert^2} d(\sqrt{\beta}\sigma_2)\dots d(\sqrt{\beta}\sigma_n) \\
\leq &\int_{(\sigma_2,\sigma_3, \dots, \sigma_n)\in (\mathbb{S}^{N-1})^{V\setminus\{1\}}} 
 4\beta \exp(-\frac{\beta}{2n^2}) d(\sqrt{\beta}\sigma_2)\dots d(\sqrt{\beta}\sigma_n),
\end{aligned}
\end{equation}
which converges to $0$. Thus it remains to prove that $\beta \mathbb{E}_{\mu^r_{G,N,\beta}}1_{\sigma_2^N>0} \cdot (1- \sigma_2^N)^2\to 0$. Note that
\begin{equation}
\begin{aligned}
\beta \mathbb{E}_{\mu^r_{G,N,\beta}}1_{\sigma_2^N>0} \cdot (1- \sigma_2^N)^2 & \leq \beta \mathbb{E}_{\mu^r_{G,N,\beta}}1_{\sigma_2^N>0} \cdot (1- (\sigma_2^N)^2)^2 \\
& \leq \beta \mathbb{E}_{\mu^r_{G,N,\beta}} \cdot (1- (\sigma_2^N)^2)^2\\
& = \beta \mathbb{E}_{\mu^r_{G,N,\beta}} \cdot (\sum_{s=1}^{N-1}(\sigma_2^s)^2)^2\\
& = (N-1)^2 \beta \mathbb{E}_{\mu^r_{G,N,\beta}} \cdot (\sigma_2^1)^4.
\end{aligned}
\end{equation}
On the other hand, take $k=2$ in (\ref{moment}) we have $\mathbb{E}_{\mu^r_{G,N,\beta}} (\sqrt{\beta}\sigma_2^1)^4 \to \mathbb{E}(\gamma_2^1)^4<+\infty$, thus 
\begin{equation}
\beta \mathbb{E}_{\mu^r_{G,N,\beta}} (\sigma_2^1)^4 \to 0,
\end{equation}
and this completes the proof when $N>1$.

For $N=1$ (Ising model), define an event $A:= \{\sigma_i= \sigma_j \text{ for all } i, j \in V\}$, we have that
\begin{equation}
\beta \cdot \mathbb{E}_{\mu_{G,N,\beta}} \lVert \sigma_x - \sigma_y \rVert^2 \leq  4\beta \mu_{G,N,\beta}(A^c).
\end{equation}
Recall that $d\sigma_i= \frac{1}{2}\delta_{-1}+ \frac{1}{2}\delta_1$ ($i\in V$) in the definition (\ref{def}) of Ising model. We note that
\begin{equation}
\sum\limits_{\sigma\in \{-1,1\}^V}1_{A^c}\cdot e^{-\frac{\beta}{2}\sum\limits_{\{i, j\}\in E}(\sigma_i - \sigma_j)^2} \leq \sum\limits_{\sigma\in \{-1,1\}^V}1_{A^c}\cdot e^{-2\beta} \leq 2^n e^{-2\beta},
\end{equation}
and
\begin{equation}
\sum\limits_{\sigma\in \{-1,1\}^V}e^{-\frac{\beta}{2}\sum\limits_{\{i, j\}\in E}(\sigma_i - \sigma_j)^2} \geq \sum\limits_{\sigma\in \{-1,1\}^V}1_{A}\cdot e^{-\frac{\beta}{2}\sum\limits_{\{i, j\}\in E}(\sigma_i - \sigma_j)^2}\geq 1.
\end{equation}
Thus $4\beta \mu_{G,N,\beta}(A^c) \leq 4\beta \cdot 2^n e^{-2\beta}$, which converges to $0$ as $\beta \to \infty$. 
\end{proof}
\begin{rem}
\label{rootrem}
By a similar argument, for $N\geq 2$, we have that
\begin{equation}
\beta \cdot \mathbb{E}_{\mu_{G,N,\beta}} \lVert \sigma_x - \sigma_y \rVert^2 = \beta \cdot \mathbb{E}_{\mu_{G,N,\beta}}^r \lVert \sigma_x - \sigma_y \rVert^2 \to (N-1)\mathbb{E}(\gamma_x^1-\gamma_y^1)^2=(N-1)(\frac{u_{x1}(x)}{d(x)}+\frac{u_{y1}(y)}{d(y)}-2\frac{u_{x1}(y)}{d(x)}).
\end{equation}
(This equation also works for the case $N=1$.)

Of course, we could further conclude that for a general connected graph $G$ and $x,y,z\in G$:
\begin{equation}
\frac{u_{xy}(x)}{d(x)}= \frac{u_{xz}(x)}{d(x)}+\frac{u_{yz}(y)}{d(y)}-2\frac{u_{xz}(y)}{d(x)},
\end{equation}
and this identity can also be verified by more elementary approaches.
\end{rem}
\begin{thm}
\label{Wbound}
Given $N\geq 2$ and $\delta>0$, there exists a constant $c=c(N, \delta)>0$, such that for any connected graph $G=(V, E)$ where $|V|=n$ and $|E|=m$ we have $\mathbb{E}_{\mu_{G,N,\beta}} \sigma_x \cdot \sigma_y>1-\delta$ for $\beta>c(n^2 \log m+ n m)$ and $x,y \in V$.
\end{thm}
\begin{proof}
As before, we take $x=2$ and $y=1$, and we consider the spin $O(N)$ model with $\sigma_1$ rooted to be the north point $(0, \dots, 0, 1)$. Note that there exists $\epsilon>0$ such that for any $\sigma_2^N > 1-\epsilon$ we have $(\sigma_1-\sigma_2)^2<\delta$, thus
\begin{equation}
\mathbb{E}_{\mu_{G,N,\beta}^r} 1_{\sigma_2^N> 1-\epsilon} \cdot \lVert\sigma_1 - \sigma_2 \rVert^2 <\delta.
\end{equation}
On the other hand, let $\chi_0=1, \chi_1, \dots, \chi_{p-1}, \chi_p=2$  be the shortest the path on $G$ from $1$ to $2$, so that $p\leq n$. If $\sigma_2^N\leq 1-\epsilon$, by Cauchy–Schwarz inequality we have
\begin{equation}
\begin{aligned}
e^{-\frac{1}{2}\sum\limits_{\{i, j\}\in E}\lVert\sqrt{\beta}\sigma_i-\sqrt{\beta}\sigma_j\rVert^2}\leq & e^{-\frac{\beta}{2}\sum\limits_{i=1}^{p} (\sigma_{\chi_i}^N-\sigma_{\chi_{i-1}}^N)^2}\\
\leq & e^{-\frac{\beta}{2p}(\sum\limits_{i=1}^{p} (\sigma_{\chi_i}^N-\sigma_{\chi_{i-1}}^N))^2}\\
\leq &\exp(-\frac{\beta\epsilon^2}{2n}).
\end{aligned}
\end{equation}
Thus
\begin{equation}
\label{top}
\begin{aligned}
&\int_{(\sigma_2,\sigma_3, \dots, \sigma_n)\in (\mathbb{S}^{N-1})^{V\setminus\{1\}}} 
1_{\sigma_2^N\leq 1-\epsilon} \cdot \lVert\sigma_1 - \sigma_2 \rVert^2 \cdot e^{-\frac{1}{2}\sum\limits_{\{i, j\}\in E}\lVert\sqrt{\beta}\sigma_i-\sqrt{\beta}\sigma_j\rVert^2} d(\sqrt{\beta}\sigma_2)\dots d(\sqrt{\beta}\sigma_n) \\
\leq &\int_{(\sigma_2,\sigma_3, \dots, \sigma_n)\in (\mathbb{S}^{N-1})^{V\setminus\{1\}}} 
 4 \exp(-\frac{\beta\epsilon^2}{2n}) d(\sqrt{\beta}\sigma_2)\dots d(\sqrt{\beta}\sigma_n)\\
 = &  4(C_1\beta^{\frac{N-1}{2}})^{n-1}\cdot \exp(-\frac{\beta\epsilon^2}{2n}),
\end{aligned}
\end{equation}
where $C_1$ is a constant related to the surface area of spheres. Furthermore,
\begin{equation}
\label{bottom}
\begin{aligned}
&\int_{(\sigma_2,\sigma_3, \dots, \sigma_n)\in (\mathbb{S}^{N-1})^{V\setminus\{1\}}} 
 e^{-\frac{1}{2}\sum\limits_{\{i, j\}\in E}\lVert\sqrt{\beta}\sigma_i-\sqrt{\beta}\sigma_j\rVert^2} d(\sqrt{\beta}\sigma_2)\dots d(\sqrt{\beta}\sigma_n) \\
\geq &\int_{(\sigma_2,\sigma_3, \dots, \sigma_n)\in (\mathbb{S}^{N-1}\cap B(\sigma_1, \frac{1}{2\sqrt{\beta}}))^{V\setminus\{1\}}}
  e^{-\frac{1}{2}\sum\limits_{\{i, j\}\in E}\lVert\sqrt{\beta}\sigma_i-\sqrt{\beta}\sigma_j\rVert^2} d(\sqrt{\beta}\sigma_2)\dots d(\sqrt{\beta}\sigma_n)\\
\geq &  \int_{(\sigma_2,\sigma_3, \dots, \sigma_n)\in (\mathbb{S}^{N-1}\cap B(\sigma_1, \frac{1}{2\sqrt{\beta}}))^{V\setminus\{1\}}}
  e^{-\frac{1}{2} m} d(\sqrt{\beta}\sigma_2)\dots d(\sqrt{\beta}\sigma_n)\\
\geq & C_2^{n-1} \cdot \exp(-\frac{1}{2}m),
\end{aligned}
\end{equation}
where $C_2$ is a constant related to a lower bound of the surface area of $\sqrt{\beta}\mathbb{S}^{N-1}\cap B(\sqrt{\beta}\sigma_1, \frac{1}{2})$. We take $c=c(N,d)$ large enough to ensure that there exists such a lower bound as the area of $\sqrt{\beta}\mathbb{S}^{N-1}\cap B(\sqrt{\beta}\sigma_1, \frac{1}{2})$ naturally converges to the area of a (N-1)-dimensional disc when $\beta\to \infty$.

Recall that $\mu_{G,N,\beta}^r$ is the measure of the rooted spin $O(N)$ model. From (\ref{top}) and (\ref{bottom}) we conclude that there is $C_0=C_0(N,\delta)>0$ large enough such that
\begin{equation}
\begin{aligned}
\mathbb{E}_{\mu_{G,N,\beta}^r} 1_{\sigma_2^N\leq 1-\epsilon} \cdot \lVert\sigma_1 - \sigma_2 \rVert^2\leq & \frac{4(C_1\beta^{\frac{N-1}{2}})^{n-1}\cdot\exp(-\frac{\beta\epsilon^2}{2n})}{C_2^{n-1} \cdot \exp(-\frac{1}{2}m)}\\
\leq & C_0^{n} \beta^{C_0 n}\exp(-\frac{\beta}{C_0 n}+C_0 m).
\end{aligned}
\end{equation}
Note also that $m\geq n-1 \geq \frac{1}{2} n$ as the graph is connected. Thus, for $\beta= \kappa(n^2 \log m+ n m)\leq 6\kappa m^3$ where $\kappa>0$, we have 
\begin{equation}
\begin{aligned}
-\log \mathbb{E}_{\mu_{G,N,\beta}^r} 1_{\sigma_2^N> 1-\epsilon}\cdot \lVert\sigma_1 - \sigma_2 \rVert^2 &\geq \frac{\beta}{C_0 n}-C_0 n\log \beta-C_0 m-n\log C_0 \\
& \geq \frac{ \kappa(n \log m+m)}{C_0}-C_0 n \log (6 \kappa m^3) -3 C_0 m \\
& \geq (\frac{ \kappa}{C_0}-3C_0)n\log m +(\frac{ \kappa}{C_0}-3C_0) m -C_0  (\log 6\kappa) \cdot n,
\end{aligned}
\end{equation}
which will be larger than $-\log \delta$ if we take $\kappa>c(N,\delta)$ to be large enough. So we have 
$$
\mathbb{E}_{\mu_{G,N,\beta}^r} \lVert \sigma_x - \sigma_y \rVert^2= \mathbb{E}_{\mu_{G,N,\beta}^r} 1_{\sigma_2^N\leq 1-\epsilon}\cdot \lVert\sigma_1 - \sigma_2 \rVert^2 +\mathbb{E}_{\mu_{G,N,\beta}^r} 1_{\sigma_2^N> 1-\epsilon} \cdot \lVert\sigma_1 - \sigma_2 \rVert^2<2\delta,
$$
and $\mathbb{E}_{\mu_{G,N,\beta}} \sigma_x \cdot \sigma_y= 1-\frac{1}{2}\mathbb{E}_{\mu_{G,N,\beta}^r} \lVert \sigma_x - \sigma_y \rVert^2>1-\delta$.
\end{proof}

We now proceed to prove Theorem~\ref{low}, establishing Gaussian domination in the low-temperature regime.
\begin{thm}
Given $N\geq 1$ and a connected graph $G=(V, E)$, there exists $\beta_d>0$, such that there is a Gaussian domination in the spin $O(N)$ model for $G$ and $\beta>\beta_d$.
\end{thm}
\begin{proof}
We rewrite (\ref{de}) for convenience. Recall that $v= (v_1, v_2, \dots, v_n) \in (\mathbb{R}^N)^V$ ($v \neq 0$) is a direction vector. Note that by the translation invariance of the modified partition function $Z^*_{G, N, \beta}$, we could suppose $v_x (x\in V)$ are not all the same and furthermore renormalize $v$ to take $v_1=0$ and $\sum\limits_{x\in V\setminus \{1\}} \lVert v_x \rVert ^2=1$.
\begin{equation}
\begin{aligned}
Z_v''(0)=&  \beta^2\int_{(\mathbb{S}^{N-1})^V}[\sum_{\{i, j\}\in E}(v_i-v_j)\cdot(\sigma_i-\sigma_j)]^2\exp(-\frac{\beta}{2}\sum\limits_{\{i, j\}\in E}\lVert\sigma_i-\sigma_j \rVert^2) d\sigma \\
&-\beta\int_{(\mathbb{S}^{N-1})^V}[\sum_{\{i, j\}\in E}\lVert v_i-v_j \rVert^2]\exp(-\frac{\beta}{2}\sum\limits_{\{i, j\}\in E}\lVert\sigma_i-\sigma_j \rVert^2) d\sigma\\
=& \beta Z^*_{G, N,\beta}(0) \cdot \mathbb{E}_{\mu_{G,N,\beta}}\big[\beta [\sum_{\{i, j\}\in E}(v_i-v_j)\cdot(\sigma_i-\sigma_j)]^2-\sum_{\{i, j\}\in E}\lVert v_i-v_j \rVert^2\big].
\end{aligned}
\end{equation}
As before, we write $v_x= (v_x^1, \dots, v_x^N)=\sum_{l=1}^N v_x^l e_l$ for $x\in V$, where \( \{e_1, e_2, \dots, e_N\} \) is the standard orthonormal basis of \( \mathbb{R}^N \). Note that for any $i,j \in V$, $l_1, l_2 \in \{1, 2, \dots, N\}$ and $l_1 \neq l_2$ we have $\mathbb{E}_{\mu_{G,N,\beta}}\sigma_i^{l1}\cdot \sigma_j^{l2}=0$, as the spin $O(N)$ model is invariant under a global reflection that sends $e_{l1}$ to $-e_{l1}$. Thus
\begin{equation}
\label{regular}
\begin{aligned}    
\frac{Z_v''(0)}{\beta Z^*_{G, N,\beta}(0)}= & \sum_{l=1}^N \mathbb{E}_{\mu_{G,N,\beta}}\beta (\sum_{\{i, j\}\in E}(v_i^l-v_j^l)\cdot(\sigma_i^l-\sigma_j^l))^2-\sum_{\{i, j\}\in E}(v_i^l-v_j^l)^2\\
= &\sum_{l=1}^N \mathbb{E}_{\mu_{G,N,\beta}}\beta (\sum_{x \in V} v_x^l \cdot \Delta \sigma_x^l)^2- \sum_{\{i, j\}\in E}\lVert v_i-v_j \rVert^2\\
= &\sum_{l=1}^N  (\sum_{x,y \in V} v_x^l v_y^l \cdot \mathbb{E}_{\mu_{G,N,\beta}}\beta \Delta \sigma_x^l \cdot \Delta \sigma_y^l)- \sum_{\{i, j\}\in E}\lVert v_i-v_j \rVert^2,
\end{aligned}
\end{equation}
where $\Delta \sigma_x^l:=\sum\limits_{\{x, z\}\in E} (\sigma_x^l-\sigma_z^l)$. On the other hand, by the remark under Theorem \ref{mom}, for any $x,y \in V, l=1, 2, \dots, N$, we have when $\beta \to \infty$
\begin{equation}
\begin{aligned}    
\mathbb{E}_{\mu_{G,N,\beta}} \beta\sigma_x^l \cdot \sigma_y^l= & \frac{1}{N} \beta \cdot \mathbb{E}_{\mu_{G,N,\beta}} \sigma_x \cdot \sigma_y\\
=& \frac{1}{2N} \beta \cdot (1-\mathbb{E}_{\mu_{G,N,\beta}} \lVert\sigma_x-\sigma_y \rVert^2)\\
\to& \frac{1}{2N} (\beta-(N-1) (\frac{u_{x1}(x)}{d(x)}+\frac{u_{y1}(y)}{d(y)}-2\frac{u_{x1}(y)}{d(x)})).
\end{aligned}
\end{equation}
Thus
\begin{equation}
\begin{aligned}    
\mathbb{E}_{\mu_{G,N,\beta}} \beta\sigma_x^l \cdot \Delta\sigma_y^l \to & 0\cdot \beta +\frac{N-1}{2N} \sum_{\{y,w\}\in E} [\frac{u_{x1}(x)}{d(x)}(1-1)+\frac{u_{y1}(y)}{d(y)}-\frac{u_{w1}(w)}{d(w)}-2\frac{u_{x1}(y)}{d(x)}+2\frac{u_{x1}(w)}{d(x)}]\\
= & \frac{N-1}{2N}\sum_{\{y,w\}\in E}(\frac{u_{y1}(y)}{d(y)}-\frac{u_{w1}(w)}{d(w)})+\frac{N-1}{N}\cdot\frac{\Delta u_{x1}(y)}{d(x)},
\end{aligned}
\end{equation}
and 
\begin{equation}
\begin{aligned}    
\mathbb{E}_{\mu_{G,N,\beta}} \beta\Delta \sigma_x^l \cdot \Delta \sigma_y^l \to & \frac{N-1}{2N}\sum_{\{x,z\}\in E}\sum_{\{y,w\}\in E}[(\frac{u_{y1}(y)}{d(y)}-\frac{u_{w1}(w)}{d(w)})\cdot(1-1)]\\
&+ \frac{N-1}{N} \sum\limits_{\{x,z\}\in E}(\frac{\Delta u_{x1}(y)}{d(x)}-\frac{\Delta u_{z1}(y)}{d(z)})\\
=&\frac{N-1}{N} \sum\limits_{\{x,z\}\in E}(\frac{\Delta u_{x1}(y)}{d(x)}-\frac{\Delta u_{z1}(y)}{d(z)}).
\end{aligned}
\end{equation}
By Lemma \ref{deltalem} it's not hard to check that $\mathbb{E}_{\mu_{G,N,\beta}} \beta \Delta \sigma_x^l \cdot \Delta \sigma_y^l \to \frac{N-1}{N} (d(x) \cdot 1_{x=y}- 1_{\{x,y\} \in E})$. (The case $x=1$ or $y=1$ is not excluded.) So we have
\begin{equation}
\label{limiting}
\begin{aligned}    
\frac{Z_v''(0)}{\beta Z^*_{G, N,\beta}(0)}\to & \frac{N-1}{N} \sum_{l=1}^N ( \sum_{x\in V}d(x)(v_x^l)^2- 2\sum_{\{i,j\} \in E}v_i^l v_j^l)- \sum_{\{i, j\}\in E}\lVert v_i-v_j \rVert^2\\
= & -\frac{1}{N}\sum_{\{i, j\}\in E}\lVert v_i-v_j \rVert^2.
\end{aligned}
\end{equation}
Since $G$ is connected, by $v_1=0$ and $\sum\limits_{x\in V\setminus \{1\}} \lVert v_x \rVert ^2=1$ we have that
\begin{equation}
\label{cubeb}
\sum_{\{i, j\}\in E}\lVert v_i-v_j \rVert^2\geq (\frac{1}{n} \max_{x\in V} \lVert v_x -v_1\rVert)^2 =(\frac{1}{n} \max_{x\in V} \lVert v_x\rVert)^2  \geq \frac{1}{n^3}.
\end{equation}
For any $\epsilon >0$, we could find $\beta_d$ large enough such that for any $\beta>\beta_d$, $l\in \{1, 2, \dots, N\}$ and $x,y\in V$ we have $|\mathbb{E}_{\mu_{G,N,\beta}} \beta \Delta \sigma_x^l \cdot \Delta \sigma_y^l - \frac{N-1}{N} (d(x) \cdot 1_{x=y}- 1_{\{x,y\} \in E})|<\epsilon$. Thus, by (\ref{regular}), (\ref{limiting}) and (\ref{cubeb}) we have $\frac{Z_v''(0)}{\beta Z^*_{G, N,\beta}(0)}<-\frac{1}{N n^3}+ \epsilon \sum\limits_{l=1}^N \sum\limits_{x,y \in V} |v_x^l v_y^l| \leq -\frac{1}{N n^3}+\epsilon  N n^2 $. Take $\epsilon<\frac{1}{N^2 n^5}$, we could conclude that $Z_v''(0)<0$ and there is a Gaussian domination for $\beta>\beta_d$.
\end{proof}

\section{A Counterexample of Gaussian domination}
\label{sec5} 
In this section, we present the counterexample for Theorem \ref{addstar}. 
\begin{thm}
For any $N$ and $\beta>0$, there exists $n_0\in \mathbb{N}$, such that for any graph $G=(V, E)$, we can find another graph $G'=(V', E')$, such that $G \subset G'$, $|V'|-|V| \leq n_0$, $|E'|- |E| \leq n_0$, and there's no Gaussian domination in the spin $O(N)$ model for $G'$ and $\beta$.
\end{thm}
\begin{proof} We begin with the proof of the following lemma.
\begin{lem}
\label{star'}
For any $N$ and $\beta>0$, a graph $G_\beta=(V_\beta, E_\beta)$ always exists such that there is no Gaussian domination in the spin $O(N)$ model for $G_\beta$ and $\beta$.
\end{lem}
\begin{figure}[htbp]
    \centering
    
    \includegraphics[scale=1]{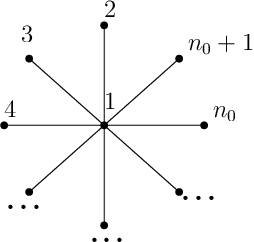}
    \caption{Counterexample}
    \label{fig:example}
\end{figure}
\begin{proof}[Proof of Lemma \ref{star'}]
Pick a graph $G_\beta=(V_\beta,E_\beta)$ such that $|V_\beta|=n=n_0+1, n_0 \in \mathbb{N}^*$ and $\{i,j\}\in E_\beta$ if and only if $i=1$ or $j=1$. We will prove that $Z^*_{G_\beta, N,\beta} (0)$ is not the maximum when $n_0$ is large enough. 

Let $e$ be a unit vector of $\mathbb{R}^N$, we fix a direction vector $v=(v_1, v_2 \dots, v_n) \in (\mathbb{R}^N)^V$ by setting $v=(e,0,\dots, 0)$. Again we define $Z_v(\eta) : \mathbb{R} \to \mathbb{R}$ by $Z_v(\eta):= Z^*_{G_\beta, N, \beta}(\eta v)$ for any $\eta \in \mathbb{R}$. Furthermore, for convenience we write $U(\sigma_1, \sigma_2):= \exp(-\frac{1}{2}\beta\lVert\sigma_1- \sigma_2 \rVert^2)$ and $U(\sigma_1)=U(0,\sigma_1)$ for short. Note that by symmetry $Z_v'(0)=0$, we want to show that $Z_v''(0)>0$ for $n_0$ large enough. 

In fact,
\begin{equation}
\begin{aligned}
 Z_v(\eta)&= \int_{\sigma} \prod\limits_{\{i, j\}\in E} U(\sigma_i +\eta v_i, \sigma_j + \eta v_j) \cdot d\sigma \\
 & =\int_{\sigma_1} \prod_{i\in V\setminus \{1\}}(\int_{ \sigma_i} U(\sigma_1 +\eta e - \sigma_i)d\sigma_i) \cdot d\sigma_1\\
 & =\int_{\sigma_1} (\int_{ \sigma_2} U(\sigma_1 +\eta e - \sigma_2)d\sigma_2)^{n_0} \cdot d\sigma_1.
 \end{aligned}
\end{equation}
Thus, by the dominated convergence theorem we have
\begin{equation}
\begin{aligned}
 Z_v'(\eta) &=n_0\int_{\sigma_1} (\int_{ \sigma_2 }  \frac{\partial}{\partial e}U(\sigma_1 +\eta e- \sigma_2)d\sigma_2 )(\int_{\sigma_2} U(\sigma_1 +\eta e- \sigma_2)d\sigma_2 )^{n_0-1}\cdot d\sigma_1,\\
 \end{aligned}
\end{equation}
and

\begin{equation}
\begin{aligned}
 Z_v''(0) &=n_0(n_0-1)\int_{\sigma_1} (\int_{ \sigma_2 }  \frac{\partial}{\partial e}U(\sigma_1- \sigma_2)d\sigma_2 )^{2}(\int_{\sigma_2 } U(\sigma_1- \sigma_2)d\sigma_2 )^{n_0-2}\cdot d\sigma_1\\
 &+n_0\int_{\sigma_1} (\int_{ \sigma_2 }  \frac{\partial^2}{\partial e^2}U(\sigma_1- \sigma_2)d\sigma_2 )(\int_{\sigma_2 } U(\sigma_1- \sigma_2)d\sigma_2 )^{n_0-1}\cdot d\sigma_1 .
 \end{aligned}
\end{equation}
Note that by symmetry $\int_{\sigma_2} U(\sigma_1- \sigma_2)d\sigma_2$ is a positive constant for any $\sigma_1 \in  S^{N-1}$,  and we denote this constant by $C$. Thus, we only need to check that 
\begin{equation}
\label{pos}    
\int_{\sigma_1} (\int_{ \sigma_2}  \frac{\partial}{\partial e}U(\sigma_1- \sigma_2)d\sigma_2 )^{2}>0,
\end{equation}
so that
\begin{equation}
\label{sod}
\begin{aligned}
\frac{1}{n_0 C^{n_0-2}} Z_v''(0) = & (n_0-1)\int_{\sigma_1 } (\int_{ \sigma_2 }  \frac{\partial}{\partial e}U(\sigma_1- \sigma_2)d\sigma_2 )^{2}\cdot d\sigma_1 \\
&+C\int_{\sigma_1} (\int_{ \sigma_2 }  \frac{\partial^2}{\partial e^2}U(\sigma_1- \sigma_2)d\sigma_2 )\cdot d\sigma_1
\end{aligned}
\end{equation}
would be positive for $n_0$ large enough.

On the other hand, when $\sigma_1=e$ we could check that 
$$\int_{ \sigma_2}\frac{\partial}{\partial e}U(\sigma_1- \sigma_2)d\sigma_2 \neq 0,$$
thus (\ref{pos}) holds as it is an integral of a continuous non-negative function of $\sigma_1$.

\end{proof}

Having established the lemma, we now turn to proving the theorem itself. We simply take $V'= V \cup \{n+1, n+2, \dots, n+n_0\}$ and $E'= E \cup \{\{1,n+1\}, \{1,n+2\}, \dots, \{1,n+n_0\}\}$, $v_1=v_2=\dots =v_n=e$, and $v_{n+1}=v_{n+2}=\dots =v_{n+n_0}=0$. $e$ is again a unit vector in $\mathbb{R}^N$. In this case, we note that 
\begin{equation}
\begin{aligned}    
Z^*_{G', N, \beta}(\eta v)=&\int_{\sigma}\prod_{\{x,y\}\in E} U(\sigma_x -\sigma_y) \prod_{\{1,y\}\in E'} U(\sigma_1+\eta e -\sigma_y) d\sigma \\ 
=&\int_{\sigma_1}(\int_{\sigma_2}\dots\int_{\sigma_n}\prod_{\{x,y\}\in E} U(\sigma_x -\sigma_y)d\sigma_2\dots d\sigma_n)\cdot\\
&(\int_{\sigma_{n+1}}\dots\int_{\sigma_{n+n_0}} \prod_{\{1,y\}\in E'} U(\sigma_1+\eta e -\sigma_y) d\sigma_{n+1}\dots d\sigma_{n+n_0})d\sigma_1\\
=&\int_{\sigma_1}\frac{1}{\int_{\sigma_1}d\sigma_1}Z^*_{G, N, \beta}(0) (\int_{ \sigma_{n+1}} U(\sigma_1 +\eta e- \sigma_{n+1}) \cdot d\sigma_{n+1})^{n_0} d\sigma_1 \\
=&\frac{1}{\int_{\sigma_1}d\sigma_1}Z^*_{G, N, \beta}(0) \int_{\sigma_1}(\int_{ \sigma_{n+1}} U(\sigma_1 +\eta e- \sigma_{n+1}) \cdot d\sigma_{n+1})^{n_0}d\sigma_1.
\end{aligned}
\end{equation}
This is simply a constant multiple of the partition function in Lemma \ref{star'}. By the computation there we know that there is no Gaussian domination for a suitable $n_0$.
\end{proof}
\begin{rem}
Note that in (\ref{sod}) we actually have $$\int_{\sigma_1} (\int_{ \sigma_2 }  \frac{\partial^2}{\partial e^2}U(\sigma_1- \sigma_2)d\sigma_2 )\cdot d\sigma_1<0.$$ This is, in fact, the Gaussian domination for a segment (or a torus with side length $2$ and dimension $1$).
\end{rem}

\section{A Counterexample  in binary trees}
\label{sec6}
In this section, we present a counterexample to Gaussian domination in binary trees, thereby establishing Theorem~\ref{binary}. For \( N = 1 \) (the Ising model), we use, for convenience, the spin measure \( d\sigma_i = \delta_{-1} + \delta_1 \) throughout this section.
\begin{thm}
Let $N=1$, there exists $\beta_0>0$ such that for any $\beta>\beta_0$, there exists a binary tree $G=(V, E)$, such that there is no Gaussian domination in the spin $O(N)$ model on $G$ at inverse temperature $\beta$.
\end{thm}
\begin{figure}[htbp]
    \centering
    
    \includegraphics[scale=1]{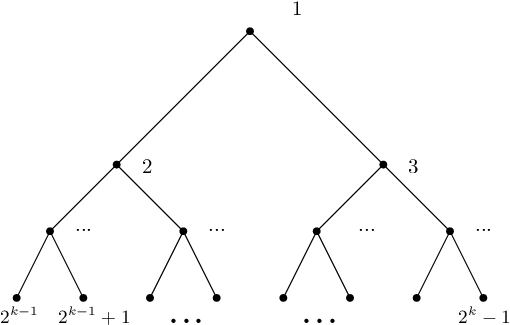}
    \caption{Binary Tree}
    \label{fig:tree}
\end{figure}
\begin{proof}
We simply let $G=(V, E)$ be a perfect binary tree with depth $k>3$ so that $|V|=2^k-1$. For convenience, we denote the first $k-1$ generations of $G$ by $G'=(V',E')$, so that $|V'|=2^{k-1}-1$. Take a direction vector $v=(v_1, v_2, \dots, v_{2^k-1})\in \mathbb{R}^V$, where $v_1=v_2= \dots=v_{2^{k-1}-1}=0$, $v_{2^{k-1}}=v_{2^{k-1}+1}=\dots=v_{2^k-1}=1$. Recall that $Z_v(\eta) : \mathbb{R} \to \mathbb{R}$ is defined by $Z_v(\eta):= Z^*_{G, N, \beta}(\eta v)$ for any $\eta \in \mathbb{R}$. Note that 
\begin{equation}
\begin{aligned}    
Z_v(\eta) =& \sum_{\sigma\in\{-1, 1\}^V}\exp(-\frac{\beta}{2}\sum_{\{i, j\}\in E}\lVert\sigma_i+\eta v_i-\sigma_j-\eta v_j \rVert^2)\\
                     =& \sum_{\sigma\in\{-1, 1\}^V}Z^*_{G',N,\beta}(0)\cdot \mu_{G',N,\beta}(\sigma_1,\sigma_2,\dots,\sigma_{2^{k-1}-1})\exp(-\frac{\beta}{2}\sum_{\{i, j\}\in E\setminus E'}\lVert\sigma_i+\eta v_i-\sigma_j-\eta v_j \rVert^2)\\
                     =& \sum_{\sigma\in\{-1, 1\}^{V'}}Z^*_{G',N,\beta}(0)\cdot \mu_{G',N,\beta}(\sigma_1,\sigma_2,\dots,\sigma_{2^{k-1}-1})\cdot\\
                     &\prod_{i=2^{k-2}}^{2^{k-1}-1}(\exp(-\frac{\beta}{2}\eta^2)+1_{\sigma_i=1}\cdot\exp(-\frac{\beta}{2}(2-\eta)^2)+1_{\sigma_i=-1}\cdot\exp(-\frac{\beta}{2}(2+\eta)^2))^2\\
                     =& \sum_{\sigma\in\{-1, 1\}^{V'}}Z^*_{G',N,\beta}(0)\cdot \mu_{G',N,\beta}(\sigma_1,\sigma_2,\dots,\sigma_{2^{k-1}-1})\prod_{i=2^{k-2}}^{2^{k-1}-1}\mathcal{U}(\sigma_i \eta)^2,                     
\end{aligned}
\end{equation}
where $\mathcal{U}(\eta):=\exp(-\frac{\beta}{2}\eta^2)+\exp(-\frac{\beta}{2}(2-\eta)^2)$.
Furthermore, notice that when $\eta \to 0$,
\begin{equation}
\mathcal{U}(\eta)=\mathcal{U}(0)+\mathcal{U}'(0)\eta+\frac{1}{2}\mathcal{U}''(0)\eta^2+o(\eta^2),
\end{equation}
so that
\begin{equation}
\mathcal{U}(\sigma_i \eta)^2=\mathcal{U}(0)^2+2\mathcal{U}(0)\mathcal{U}'(0)\sigma_i\eta+(\mathcal{U}(0)\mathcal{U}''(0)+\mathcal{U}'(0)^2)\eta^2+o(\eta^2).
\end{equation}
Thus we have
\begin{equation}
\label{Tal}
\begin{aligned}    
\frac{Z_v(\eta)}{Z^*_{G',N,\beta}(0)} = & \sum_{\sigma\in\{-1, 1\}^{V'}}
                     \mu_{G',N,\beta}(\sigma) \prod_{i=2^{k-2}}^{2^{k-1}-1}(\mathcal{U}(0)^2+2\mathcal{U}(0)\mathcal{U}'(0)\sigma_i\eta+(\mathcal{U}(0)\mathcal{U}''(0)+\mathcal{U}'(0)^2)\eta^2+o(\eta^2))\\
                     =& \frac{Z_{G, N, \beta}}{Z^*_{G',N,\beta}(0)}+ \eta^2 \sum_{\sigma\in\{-1, 1\}^{V'}}\mu_{G',N,\beta}(\sigma)\cdot2^{k-2}\mathcal{U}(0)^{2(2^{k-2}-1)}(\mathcal{U}(0)\mathcal{U}''(0)+\mathcal{U}'(0)^2)\\
                     & +\eta^2 \sum_{\sigma\in\{-1, 1\}^{V'}}\mu_{G',N,\beta}(\sigma)\cdot 4\mathcal{U}(0)^{2(2^{k-2}-1)}\mathcal{U}'(0)^2(\sum_{2^{k-2}\leq i<j\leq2^{k-1}-1}\sigma_i\sigma_j)+o(\eta^2).
\end{aligned}
\end{equation}
Here we only need to consider the second-order term as $Z_v(0)=Z_{G, N, \beta}$ and $Z_v'(0)=0$ by symmetry. 

On the other hand, note that 
\begin{equation}
\label{Tot}
\sum\limits_{\sigma\in\{-1, 1\}^{V'}}\mu_{G',N,\beta}(\sigma)=1
\end{equation} 
and 
\begin{equation}
\label{expect}
\sum\limits_{\sigma\in\{-1, 1\}^{V'}}\mu_{G',N,\beta}(\sigma)\sigma_i\sigma_j=\mathbb{E}_{G', N,\beta}\sigma_i\sigma_j. 
\end{equation}
Let $d(i,j)$ denote the graph distance between vertices $i$ and $j$ and let $\chi_0=i, \chi_1, \chi_2 \dots, \chi_{d(i,j)}=j$ denote the unique shortest path connecting $i$ to $j$. For convenience, we define the spin ratio $\omega_s:=\frac{\sigma_{\chi_s}}{\sigma_{\chi_{s+1}}}\in \{-1,1\}$ for $s=0, 1, \dots, d(i,j)-1$ and we have 
\begin{equation}
\begin{aligned}    
\mathbb{E}_{G', N, \beta} \sigma_i \sigma_j=& \frac{\sum_\sigma\prod_{s=0}^{d(i,j)-1}\sigma_{\chi_s}\sigma_{\chi_s+1}\exp(-\frac{\beta}{2}\lVert\sigma_{\chi_s}-\sigma_{\chi_{s+1}} \rVert^2)}{\sum_\sigma\prod_{s=0}^{d(i,j)-1}\exp(-\frac{\beta}{2}\lVert\sigma_{\chi_s}-\sigma_{\chi_{s+1}} \rVert^2)}\\
=& \frac{\sum_\sigma\prod_{s=0}^{d(i,j)-1}\omega_{s}\exp(-\frac{\beta}{2}(1-\omega_{s})^2)}{\sum_\sigma\prod_{s=0}^{d(i,j)-1}\exp(-\frac{\beta}{2}(1-\omega_s)^2)}\\
=& \frac{\prod_{s=0}^{d(i,j)-1}(\sum_\omega\omega_{ s}\exp(-\frac{\beta}{2}(1-\omega_{ s})^2))}{\prod_{s=0}^{d(i,j)-1}(\sum_\omega \exp(-\frac{\beta}{2}(1-\omega_s)^2)}\\
=&(\frac{1-\exp(-2\beta)}{1+\exp(-2\beta)})^{d(i,j)}.
\end{aligned}    
\end{equation}
Thus
\begin{equation}
\begin{aligned}
\sum_{2^{k-2}\leq i<j\leq2^{k-1}-1}\mathbb{E}_{G', N, \beta} \sigma_i \sigma_j=& \sum_{2^{k-2}\leq i<j\leq2^{k-1}-1}(\frac{1-\exp(-2\beta)}{1+\exp(-2\beta)})^{d(i,j)}\\
=&\sum_{2^{k-2}\leq i \leq2^{k-1}-1}\sum_{l=1}^{k-2}2^{l-1}(\frac{1-\exp(-2\beta)}{1+\exp(-2\beta)})^{2l}\\
=&2^{k-3}\sum_{l=1}^{k-2}(2\cdot(\frac{1-\exp(-2\beta)}{1+\exp(-2\beta)})^2)^l.
\end{aligned}
\end{equation}
Using equations~(\ref{Tal}), (\ref{Tot}), and (\ref{expect}), along with the computation above, we readily obtain the value of \( Z_v''(0) \). We conclude that
\begin{equation}
\begin{aligned}
&\frac{1}{2Z^*_{G',N,\beta}(0)\mathcal{U}(0)^{2(2^{k-2}-1)}}\cdot Z_v''(0)\\
=&2^{k-2}(\mathcal{U}(0)\mathcal{U}''(0)+\mathcal{U}'(0)^2)+4\mathcal{U}'(0)^22^{k-3}\sum_{l=1}^{k-2}(2\cdot(\frac{1-\exp(-2\beta)}{1+\exp(-2\beta)})^2)^l.\\
=&2^{k-2}(\mathcal{U}(0)\mathcal{U}''(0)+\mathcal{U}'(0)^2)+2^{k-1}\mathcal{U}'(0)^2\sum_{l=1}^{k-2}(2\cdot(\frac{1-\exp(-2\beta)}{1+\exp(-2\beta)})^2)^l.
\end{aligned}
\end{equation}
Fix $\beta_0 > 0$ sufficiently large such that
\[
2 \left( \frac{1 - \exp(-2\beta_0)}{1 + \exp(-2\beta_0)} \right)^2 > 1.
\]
Then, since $\mathcal{U}'(0)\neq 0$, as the depth $k \to \infty$, the following divergence holds:
\[
\mathcal{U}'(0)^2 \sum_{l=1}^{k-2} \left( 2 \left( \frac{1 - \exp(-2\beta)}{1 + \exp(-2\beta)} \right)^2 \right)^l \to +\infty \quad \text{for all } \beta > \beta_0.
\]
Moreover, $Z_v''(0) > 0$ under these conditions. This implies that, for any $\beta > \beta_0$ and sufficiently large $k$, Gaussian domination fails for the perfect binary tree of depth $k$.

\end{proof}

\section{From bounded Green's function to symmetric breaking}
\label{sec7}
In this section, we present another counterexample to prove Theorem~\ref{ppath}.
\begin{thm}
Let $N = 1$, given $M>0$ and $\epsilon>0$, we do not always have a universal $\beta_c=\beta_c(N, \epsilon, M)$, such that for any  graph $G$, $\beta> \beta_c$ and $x,y \in G$ where $\frac{u_{xy}(x)}{d(x)}<M$, we have $\mathbb{E}_{\mu_{G,N,\beta}} \sigma_x \cdot \sigma_y > 1-\epsilon$.
\end{thm}
\begin{figure}[htbp]
    \centering
    
    \includegraphics[scale=1]{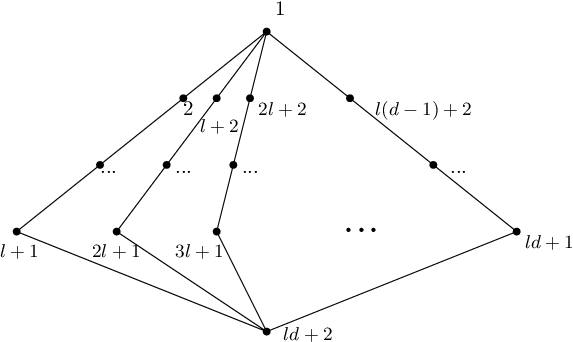}
    \caption{Parallel Paths}
    \label{fig:example}
\end{figure}
\begin{proof}
Take $M=3$, $\epsilon=\frac{1}{2}$. We construct the graph above for fixed $l=d\in \mathbb{N}^*$ so that $|V|=ld+2$. Take $x=1$ and $y=ld+2$, we first check that $\frac{u_{xy}(x)}{d(x)}$ is bounded by $M$. Note that a random walk starting from $1$ hits $ld+2$ before it goes back to $1$ is a typical gambler's ruin problem, and its probability is $\frac{1}{l+1}$. So, as the expectation of a geometric distribution, $u_{xy}(x)=\mathbb{E}_x\sum\limits_{k=1}^{H_y-1}1_{X_k=x}=\sum_{i=0}^\infty(\frac{l}{l+1})^i=l+1$. Thus $\frac{u_{xy}(x)}{d(x)}=\frac{l+1}{l}<3=M$.

On the other hand, for any $\beta>0$, note that
\begin{equation}
\begin{aligned}
\sum_{\sigma_1=\sigma_{ld+2}=1}\exp(-\frac{\beta}{2}\sum_{\{i, j\}\in E}\lVert\sigma_i-\sigma_j \rVert^2)&= (\sum_{1\leq 2k \leq l+1}\binom{l+1}{2k}\exp(-4k\beta))^d\\
&=(\frac{(1+\exp(-2\beta))^{l+1}+(1-\exp(-2\beta))^{l+1}}{2})^d.
\end{aligned}
\end{equation}
In the first equality, $2k$ is the number of edges with different values at its endpoints in one of the paths from $1$ to $ld+2$, $k\in \mathbb{N}$. Thus,
\begin{equation}
\sum_{\sigma_1 = \sigma_{ld+2}}\exp(-\frac{\beta}{2}\sum_{\{i, j\}\in E}\lVert\sigma_i-\sigma_j \rVert^2)=2(\frac{(1+\exp(-2\beta))^{l+1}+(1-\exp(-2\beta))^{l+1}}{2})^d.
\end{equation}
For the same reason,
\begin{equation}
\sum_{\sigma_1\neq \sigma_{ld+2}}\exp(-\frac{\beta}{2}\sum_{\{i, j\}\in E}\lVert\sigma_i-\sigma_j \rVert^2)=2(\frac{(1+\exp(-2\beta))^{l+1}-(1-\exp(-2\beta))^{l+1}}{2})^d.
\end{equation}
For any $\beta>0$, let $r:=\frac{1-\exp(-2\beta)}{1+\exp(-2\beta)}$, so that $0<r<1$, we will have that
\begin{equation}
\begin{aligned}
\mathbb{E}_{\mu_{G,N,\beta}} \sigma_x \cdot \sigma_y=&\frac{((1+\exp(-2\beta))^{l+1}+(1-\exp(-2\beta))^{l+1})^d-((1+\exp(-2\beta))^{l+1}-(1-\exp(-2\beta))^{l+1})^d}{((1+\exp(-2\beta))^{l+1}+(1-\exp(-2\beta))^{l+1})^d+((1+\exp(-2\beta))^{l+1}-(1-\exp(-2\beta))^{l+1})^d}\\
=&\frac{(1+r^{l+1})^d-(1-r^{l+1})^d}{(1+r^{l+1})^d+(1-r^{l+1})^d}.
\end{aligned}
\end{equation}
When $l=d \to \infty$, we have that $\log(1+r^{l+1})^d=d\log(1+r^{l+1})=dr^{l+1}+o(r^{l+1})\to 0$, thus $(1+r^{l+1})^d\to 1$. For the same reason $(1-r^{l+1})^d\to 1$, which means that
\begin{equation}
\mathbb{E}_{\mu_{G,N,\beta}} \sigma_x \cdot \sigma_y\to 0 < 1-\epsilon.
\end{equation}
This completes the proof.
\end{proof}

\end{document}